\documentclass[12pt]{amsart}

\usepackage{amssymb}
\usepackage{amsmath}
\usepackage{amsfonts}
\usepackage{graphicx}
\usepackage{amsthm}
\usepackage{enumerate}
\usepackage[mathscr]{eucal}
\usepackage{mathrsfs}
\usepackage{verbatim}

\makeatletter
\@namedef{subjclassname@2010}{%
  \textup{2010} Mathematics Subject Classification}
\makeatother

\numberwithin{equation}{section}

\theoremstyle{plain}
\newtheorem{theorem}{Theorem}[section]
\newtheorem{lemma}[theorem]{Lemma}

\theoremstyle{plain}

\numberwithin{equation}{section}

\theoremstyle{remark}

\DeclareMathOperator{\supp}{supp}
\DeclareMathOperator{\dist}{dist}
\DeclareMathOperator{\vol}{vol}

\begin{document}
\date{}
\title
[Lattice Points]{A note on lattice points in certain finite type
domains in $\mathbb{R}^d$}
\author[J.W. Guo]{Jingwei Guo}

\address{Jingwei Guo\\School of Mathematical Sciences\\
University of Science and Technology of China\\
Hefei, Anhui Province 230026, People's Republic of China}

\email{jwguo@ustc.edu.cn}

\thanks{}

\begin{abstract}
We study the Fourier transforms of indicator functions of some
special high-dimensional finite type domains and obtain estimates of
the associated lattice point discrepancy.
\end{abstract}


\maketitle
%%%%%%%%%%%%%%%%%%%%%%%%%%%%%%%%%%%%%%%%%%%%%%%%%%%%%%%%%%%%%%%%%%%%%%%%%%%%%%%%%%%%%%%%%%%%%

\section{Introduction}\label{introduction}

Let $\mathcal{B}\subset \mathbb{R}^d$ be a compact convex domain,
which contains the origin in its interior and has a smooth boundary
$\partial \mathcal{B}$. The number of lattice points $\mathbb{Z}^d$
in the dilated domain $t\mathcal{B}$ is approximately $\vol(B)t^d$
and the lattice point problem is to study the remainder,
$P_{\mathcal{B}}(t)$, in the equation
\begin{equation*}
P_{\mathcal{B}}(t)=\#(t\mathcal{B}\cap\mathbb{Z}^d)-\vol(B)t^d \quad
\textrm{for $t\geq 1$}.
\end{equation*}
For a rich theory of this problem see Kratzel's monographs
\cite{kratzel, krztzelbook2}, Huxley's \cite{huxley}, as well as a
survey article \cite{survey2004} by Ivi\'c, Kr\"atzel, K\"uhleitner,
and Nowak. The study of non-convex domains is also interesting. For
example see Nowak~\cite{nowak-torus} and Chamizo and
Raboso~\cite{chamizo2015} for study of a torus in $\mathbb{R}^3$.

If the boundary $\partial \mathcal{B}$ has points of vanishing
curvature, the problem is hard and our knowledge (especially of high
dimensions) is fragmentary. We assume $d\geq 3$ below and refer
interested readers to \cite{krztzelbook2}, \cite{survey2004}, and
the author~\cite{guo2} for the planar case which is comparatively
well understood.

Randol~\cite{randol} considered super spheres
\begin{equation*}
\mathcal{B}=\{x\in \mathbb{R}^d :
|x_1|^{\omega}+|x_2|^{\omega}+\cdots+|x_d|^{\omega}\leq  1\}
\end{equation*}
for even integer $\omega\geq 3$, and proved that
\begin{equation}\label{randol}
P_{\mathcal{B}}(t)=\bigg\{ \begin{array}{ll}
O(t^{d-2+2/(d+1)}) & \textrm{for $\omega< d+1$},\\
O(t^{(d-1)(1-1/\omega)}) & \textrm{for $\omega\geq d+1$},
\end{array}
\end{equation}
and this estimate is the best possible when $\omega\geq d+1$.
Kr\"atzel~\cite{kratzel_odd} extended this result to odd $\omega\geq
3$ and actually gave an asymptotic formula of the remainder. See
\cite{kratzel} for further results and also
M\"uller~\cite{muller2003} for related (sharp) results when $d$ is
much larger than $\omega$.

Kr\"atzel~\cite{kratzel_2002} and Kr{\"a}tzel and
Nowak~\cite{K-N-2008,K-N-2011} study a specific class of convex
domains in $\mathbb{R}^3$, the so-called bodies of pseudo
revolution. They evaluate the contribution of flat points precisely
and also of other boundary points and obtain asymptotic formulas of
$P_{\mathcal{B}}(t)$.

Many authors tried to study more general domains in $\mathbb{R}^3$.
Partial results are available in Kr\"atzel~\cite{kratzel_2000,
kratzel_2002_MM}, Popov~\cite{popov}, Peter~\cite{peter_2002}, and
Nowak~\cite{nowak_2008}, in which a variety of (somewhat
complicated) curvature assumptions are made. Concerning curvature
assumptions, Kr{\"a}tzel and Nowak write in \cite{K-N-2008,
nowak_2008}:``In general, in dimension 3 it is not at all clear how
`natural' assumptions should look like, concerning the subset of
$\partial \mathcal{B}$ where the curvature vanishes.'' The geometry
is even more complicated in high dimensions.

One concise and possibly proper curvature assumption \emph{for every
dimension} $d\geq 3$ may be the ``\emph{of finite type}'' condition
(in the sense of Bruna, Nagel, and Wainger~\cite{BNW}). A few
results are known for convex domains of finite type. The classical
method (see for example Randol~\cite{randol}) readily yields
\begin{equation}\label{classic}
P_{\mathcal{B}}(t)=O\left(t^{(d-1)\left(1-\frac{1}{(\omega-1)d+1}
\right)}\right)
\end{equation}
($\omega$ is the type of the boundary) as a consequence of the
following bound (due to \cite[Theorem B]{BNW}) of the Fourier
transform of the indicator function of $\mathcal{B}$:
\begin{equation}
|\widehat{\chi}_{\mathcal{B}}(\xi)|\lesssim |\xi|^{-1-(d-1)/\omega}.
\label{FT}
\end{equation}
The bound \eqref{classic} is crude. Iosevich, Sawyer, and
Seeger~\cite[Theorem 1.3]{I-S-S-0} gives a better one, which shows
its dependence on the multitype of $\partial \mathcal{B}$ and
follows from a finer version of \eqref{FT} (\cite[Proposition
1.2]{I-S-S-0}). It is unknown to us whether or not their bound can
be improved in general since in high dimensions the set of boundary
points with vanishing curvature may be complicated and the related
problem of estimating $\widehat{\chi}_{\mathcal{B}}$ is not easy.

 For specific finite type domains, however, Iosevich, Sawyer, and
Seeger's bound may be improved since the two difficulties mentioned
above may be overcome. For example Randol's bound \eqref{randol} for
super spheres is better when $\omega$ is not large (say, of size
$O(d)$). Motivated by examples studied by Randol, Kratzel, and
Nowak, we study a class of high-dimensional domains and prove the
following result.

\begin{theorem}\label{theorem2}
Let $\omega_1, \ldots, \omega_d\in \mathbb{N}$ be even and
\begin{equation*}
\mathcal{D}=\{x\in \mathbb{R}^d :
x_1^{\omega_1}+\cdots+x_d^{\omega_d}\leq 1\}.
\end{equation*}
Then
\begin{equation}
\begin{split}
P_{\mathcal{D}}(t)= &\sum_{S\in \mathcal{P}_1(\mathbb{N}_d)}
                  O, \Omega\left(t^{d-1-\sum\limits_{1\leq l\leq d,\, l\notin S}1/\omega_l}\right)+\\
                  &\sum_{S\in \mathcal{P}_j(\mathbb{N}_d),\, 2\leq j\leq
d}O\left(t^{d-1-\frac{j-1}{d+1}-\frac{2d}{d+1}\sum\limits_{1\leq
l\leq d, \, l\notin S}1/\omega_l}\right).
 %%\substack{1\leq l\leq d\\ l\ni S}
\end{split}\label{thm-bound1}
\end{equation}
where $\mathbb{N}_d=\{1, 2, \ldots, d\}$ and
$\mathcal{P}_j(\mathbb{N}_d)$ is the collection of all subsets of
$\mathbb{N}_d$ having $j$ elements. If $\omega=\max\{\omega_1,
\ldots, \omega_d\}$ then
\begin{equation}
P_{\mathcal{D}}(t)\lesssim
                  t^{(d-1)(1-1/\omega)}+t^{d-2+2/(d+1)}.
                  \label{thm-bound2}
\end{equation}
\end{theorem}

{\it Remarks:} In \eqref{thm-bound1}, the first sum is the
contribution of boundary points which lie on coordinate axes; the
terms for $2\leq j \leq d-1$ come from boundary points lying on
coordinate planes but not on axes; the term for $j=d$ is
$O(t^{d-2+2/(d+1)})$, due to boundary points that are not on any
coordinate plane.

Similar examples can be studied in the same way. For interesting
results on rotations of convex domains in $\mathbb{R}^d$ see
\cite{I-S-S-0}, the author~\cite{guo3}, etc.

{\it Notations:} We set $\mathbb{Z}_{*}^{d}=\mathbb{Z}^{d}\setminus
\{0\}$, and $\mathbb{R}^d_*=\mathbb{R}^d\setminus \{0\}$. The
Fourier transform of any function $f\in L^1(\mathbb{R}^d)$ is
$\widehat{f}(\xi)=\int f(x) \exp(-2\pi x\cdot \xi) \, dx$. For
functions $f$ and $g$ with $g$ taking nonnegative real values,
$f\lesssim g$ means $|f|\leq Cg$ for some constant $C$. If $f$ is
nonnegative, $f\gtrsim g$ means $g\lesssim f$. The Landau notation
$f=O(g)$ is equivalent to $f\lesssim g$. The notation $f\asymp g$
means that $f\lesssim g$ and $g\lesssim f$.

%%%%%%%%%%%%%%%%%%%%%%%%%%%%%%%%%%%%%%%%%%%%%%%%%%%%%%%%%%%%%%%%%%%%%%%%%%%%%%%%%%%%%%%%%%%%%

%%%%%%%%%%%%%%%%%%%%%%%%%%%%%%%%%%%%%%%%%%%%%%%%%%%%%%%%%%%%%%%%%%%%%%%%%%%%%%%%%%%%%%%%%%%%%

\section{Main Estimates}

Let $\mathcal{D}$ be as defined in Theorem \ref{theorem2}. If $x\in
\partial \mathcal{D}$ let $T_x$ be the affine tangent plane to
$\partial \mathcal{D}$ at $x$. Bruna, Nagel, and Wainger~\cite{BNW}
defines a ``ball''
\begin{equation*}
\tilde{B}(x, \delta)=\{y\in \partial \mathcal{D}: \dist(y,
T_x)<\delta\}
\end{equation*}
to be a cap near $x$ cut off from $\partial \mathcal{D}$ by a plane
parallel to $T_x$ at distance $\delta$ from it.

For nonzero $\xi\in \mathbb{R}^d$ let $x(\xi)$ be the unique point
on $\partial \mathcal{D}$ where the unit exterior normal is
$\xi/|\xi|$. We first prove a bound for the surface measure of
$\tilde{B}(x(\xi), |\xi|^{-1})$:

\begin{lemma}\label{lemma1} Let $0<c\leq 1$ be a constant. For
any nonzero $\xi\in \mathbb{R}^d$ with $|\xi_d|/|\xi|\geq c$ we have
\begin{equation*}
\sigma(\tilde{B}(x(\xi), |\xi|^{-1}))\lesssim \prod_{l=1}^{d-1}
\min\left( |\xi|^{-\frac{1}{\omega_l}},
\left(|\xi_l|/|\xi|\right)^{-\frac{\omega_l-2}{2(\omega_l-1)}}|\xi|^{-\frac{1}{2}}
\right),
\end{equation*}
where the implicit constant only depends on $c$ and $\mathcal{D}$,
and the minimum takes the value $|\xi|^{-1/\omega_l}$ if $\xi_l$
vanishes for some $l$.
\end{lemma}

\begin{proof}
For any fixed $\xi$ denote $x(\xi)=(a_1, \ldots, a_d)\in \partial
\mathcal{D}$. By definition $\tilde{B}(x(\xi), |\xi|^{-1})$ is the
cap near $x(\xi)$ cut off from $\partial \mathcal{D}$
\begin{equation}
x_1^{\omega_1}+\cdots+x_d^{\omega_d}=1 \label{domainD}
\end{equation}
by the plane
\begin{equation*}
\langle x-(x(\xi)-|\xi|^{-2}\xi), \ \xi \rangle=0.
\end{equation*}
Changing variables $y_i=x_i-a_i$, combining the two equations above,
and eliminating $y_d$ yield the equation
\begin{equation}
\sum_{l=1}^{d-1}(a_l+y_l)^{\omega_l}+\left(a_d-\xi_d^{-1}-\xi_d^{-1}\sum_{l=1}^{d-1}\xi_l
y_l\right)^{\omega_d}=1. \label{(d-1)domain}
\end{equation}
To estimate $\sigma(\tilde{B}(x(\xi), |\xi|^{-1}))$ it suffices to
estimate the size of the $(d-1)$-dimensional domain bounded by
\eqref{(d-1)domain} (namely, the projection of the cap onto
$\mathbb{R}^{d-1}$). Hence it suffices to show that if $(y_1,\ldots,
y_{d-1})$ satisfies \eqref{(d-1)domain} then for each $1\leq l\leq
d-1$
\begin{equation}
\max|y_l|\lesssim \min\left( |\xi|^{-\frac{1}{\omega_l}},
\left(|\xi_l|/|\xi|\right)^{-\frac{\omega_l-2}{2(\omega_l-1)}}|\xi|^{-\frac{1}{2}}
\right). \label{ineq1}
\end{equation}

Without loss of generality we only prove the case $l=1$ since other
cases are the same.

\emph{Case 1}: $a_1=0$. Then \eqref{(d-1)domain} implies that
\begin{equation}
\max|y_1|\lesssim |\xi|^{-1/\omega_1}. \label{case1}
\end{equation}
Indeed, Taylor's expansion gives
\begin{equation}
(a_l+y_l)^{\omega_l}\geq a_l^{\omega_l}+\omega_l a_l^{\omega_l-1}y_l
\quad \textrm{if $a_l\neq 0$},\label{ineq2}
\end{equation}
and
\begin{equation}
\left(a_d-\xi_d^{-1}-\xi_d^{-1}\sum_{l=1}^{d-1}\xi_l
y_l\right)^{\omega_d} \geq a_d^{\omega_d}-\omega_d a_d^{\omega_d-1}
\xi_d^{-1}\sum_{l=1}^{d-1}\xi_l y_l+O(|\xi|^{-1}),\label{ineq3}
\end{equation}
since $\omega_l$ is even and $|\xi_d|\asymp |\xi|$. Recall that
$x(\xi)$ satisfies \eqref{domainD} and note that $\xi_1=0$ since
\begin{equation}
(\omega_1 a_1^{\omega_1-1}, \ldots, \omega_d a_d^{\omega_d-1})
\parallel \xi.\label{normal}
\end{equation}
Applying these facts and the two inequalities \eqref{ineq2} and
\eqref{ineq3} to \eqref{(d-1)domain} yields \eqref{case1}, hence
\eqref{ineq1}.

\emph{Case 2}: $a_1>0$ (the negative case is similar). We assert
that: if $\max|y_1|\leq c_1 a_1$ for a sufficiently small constant
$c_1$ then \eqref{ineq1} holds with $l=1$; otherwise it still holds.

If $\max|y_1|\leq c_1 a_1$, applying \eqref{ineq2}, \eqref{ineq3},
and
\begin{equation*}
(a_1+y_1)^{\omega_1}=a_1^{\omega_1}+\omega_1 a_1^{\omega_1-1}y_1+
a_1^{\omega_1-2}y_1^2(\omega_1(\omega_1-1)/2+O(c_1))+y_1^{\omega_1}
\end{equation*}
to \eqref{(d-1)domain} (like what we did in Case 1) yields
\begin{equation}
|y_1|\lesssim \min\left( |\xi|^{-\frac{1}{\omega_1}},
|a_1|^{-\frac{\omega_1-2}{2}}|\xi|^{-\frac{1}{2}}
\right)\label{ineq4}
\end{equation}
if $c_1$ is sufficiently small. Note that \eqref{normal} implies
$|a_1|\asymp (|\xi_1|/|\xi|)^{1/(\omega_1-1)}$. Hence the first
assertion follows immediately.

If $y_1>0$ or $y_1<0$ but $\max_{y_1<0}|y_1|\leq c_1 a_1$, a similar
argument as above proves the desired bound.

Hence, to prove the second assertion, we may assume $y_1<0$ and
$\max_{y_1<0}|y_1|>c_1 a_1$. By a compactness argument there exists
a constant $C_1$ (depending only on $c_1$ and $\mathcal{D}$) such
that $\tilde{B}(x(\xi), C_1|\xi|^{-1})$ intersects the coordinate
plane $x_1=0$. Let $P$ be a point of the intersection. Then the cap
$\tilde{B}(P, C_1|\xi|^{-1})$ intersects $\tilde{B}(x(\xi),
C_1|\xi|^{-1})$. By \cite[Theorem A]{BNW} there exists a constant
$C_2=C_2(\mathcal{D})$ such that
\begin{equation*}
\tilde{B}(x(\xi), C_1|\xi|^{-1})\subset \tilde{B}(P, C_2
C_1|\xi|^{-1}).
\end{equation*}
By this inclusion and the result of Case 1 (applying  to
$\tilde{B}(P, C_2 C_1|\xi|^{-1})$), we get
\begin{equation}
\max |y_1|\lesssim |\xi|^{-1/\omega_1}.\label{ineq5}
\end{equation}
Hence $a_1\lesssim |\xi|^{-1/\omega_1}$, which implies
\begin{equation}
|\xi|^{-\frac{1}{\omega_1}}\lesssim
\left(|\xi_1|/|\xi|\right)^{-\frac{\omega_1-2}{2(\omega_1-1)}}|\xi|^{-\frac{1}{2}}.\label{ineq6}
\end{equation}
The inequalities \eqref{ineq5} and \eqref{ineq6} gives \eqref{ineq1}
with $l=1$. This finishes the proof.
\end{proof}

By the Gauss-Green formula, \cite[Theorem B]{BNW}, and Lemma
\ref{lemma1}, we immediately get the following bound of the Fourier
transform of the indicator function of $\mathcal{D}$:
\begin{theorem}\label{theorem3}
Let $0<c\leq 1$ be a constant.  For any $\xi\in S^{d-1}$ with
$|\xi_d|\geq c$ and $\lambda>0$ we have
\begin{equation}
|\widehat{\chi}_{\mathcal{D}}(\lambda \xi)|\lesssim
\lambda^{-1}\prod_{l=1}^{d-1}
\min\left(\lambda^{-\frac{1}{\omega_l}},
|\xi_l|^{-\frac{\omega_l-2}{2(\omega_l-1)}}\lambda^{-\frac{1}{2}}
\right),
\end{equation}
where the implicit constant only depends on $c$ and $\mathcal{D}$,
and the minimum takes the value $\lambda^{-1/\omega_l}$ if $\xi_l$
vanishes for some $l$.
\end{theorem}

{\it Remark:} This Lemma is a generalization of \cite[II, Theorem
2]{randol}. Our proof is simpler due to an application of the
results from \cite{BNW}. The result shows that the size of
$|\widehat{\chi}_{\mathcal{D}}(\lambda \xi)|$ may depend on both the
size of $\lambda$ and the direction of $\xi$.

%%%%%%%%%%%%%%%%%%%%%%%%%%%%%%%%%%%%%%%%%%%%%%%%%%%%%%%%%%%%%%%%%%%%%%%%%%%%%%%%%%%

\section{Proof of Theorem \ref{theorem2}}
Let $0\leq \rho\in C_0^{\infty}(\mathbb{R}^d)$ be such that
$\int_{\mathbb{R}^d}\rho(y) \, dy=1$, $\varepsilon>0$,
$\rho_{\varepsilon}(y)= \varepsilon^{-d}\rho(\varepsilon^{-1}y)$,
and
\begin{equation*}
N_{\varepsilon}(t)=\sum_{k\in \mathbb{Z}^d}
\chi_{t\mathcal{D}}*\rho_{\varepsilon}(k),
\end{equation*}
where $\chi_{t\mathcal{D}}$ denotes the characteristic function of
$t\mathcal{D}$. By the Poisson summation formula we have
\begin{equation}
N_{\varepsilon}(t)=t^d\sum_{k\in \mathbb{Z}^d}
\widehat{\chi}_{\mathcal{D}}(tk)\widehat{\rho}(\varepsilon
k)=\textrm{vol}(\mathcal{D}) t^d+R_{\varepsilon}(t),\label{asymp}
\end{equation}
where
\begin{equation*}
R_{\varepsilon}(t)=t^d\sum_{k\in \mathbb{Z}_{*}^{d}}
\widehat{\chi}_{\mathcal{D}}(tk)\widehat{\rho}(\varepsilon k).
\end{equation*}

To estimate $R_{\varepsilon}(t)$, by using a partition of unity, we
decompose it as the sum of $S_j$, $1\leq j\leq d$, where
\begin{equation*}
S_j=t^d \sum_{k\in \mathbb{Z}_{*}^{d}} \Omega_j(k)
\widehat{\chi}_{\mathcal{D}}(tk)\widehat{\rho}(\varepsilon k)
\end{equation*}
with each $\Omega_j$ supported in $\Gamma_j=\{x\in \mathbb{R}^d :
|x_j|\geq (2d)^{-1/2}|x|\}$ and smooth away from the origin. We then
split $S_j$ (as follows) depending on the number of nonzero
components of $k$:
\begin{equation*}
S_j=\sum_{i=1}^{d}S_{i, j}
\end{equation*}
with
\begin{equation*}
S_{i, j}=t^d \sum_{(i)} \Omega_j(k)
\widehat{\chi}_{\mathcal{D}}(tk)\widehat{\rho}(\varepsilon k),
\end{equation*}
where the summation is over all $k\in \mathbb{Z}_{*}^{d}\cap
\supp(\Omega_j)$ having exactly $i$ nonzero components.

Without loss of generality we may assume $j=d$, namely, restrict the
$k\in \mathbb{Z}_*^d$ to a cone about $x_d$-axis. By Theorem
\ref{theorem3} we have
\begin{equation}
|S_{1, d}|\lesssim t^{d-1-\sum_{l=1}^{d-1}1/\omega_l} \sum_{k_d \in
\mathbb{Z}_{*}^{1}} |k_d|^{-1-\sum_{l=1}^{d-1}1/\omega_l}\lesssim
t^{d-1-\sum_{l=1}^{d-1}1/\omega_l}.\label{bound1}
\end{equation}
For $2\leq i\leq d$ we apply Theorem \ref{theorem3}, compare the
sums with integrals in polar coordinates, and get
\begin{equation*}
|S_{i, d}|\lesssim \sum_{S\in \mathcal{P}_i(\mathbb{N}_d) : d\in S}
t^{d-\frac{i+1}{2}-\sum\limits_{1\leq l\leq d, \, l\notin
S}1/\omega_l}\left(1+\varepsilon^{-\frac{i-1}{2}+\sum\limits_{1\leq
l\leq d, \, l\notin S}1/\omega_l}\right).
\end{equation*}
Note that the first term of the summand above is not larger than the
right side of \eqref{bound1}.

By using the two bounds above and similar results for other $j$'s we
obtain a bound of $R_{\varepsilon}(t)$. Note that
\begin{equation}
N_{\varepsilon}(t-C\varepsilon)\leq
\#(t\mathcal{D}\cap\mathbb{Z}^d)=\sum_{k\in \mathbb{Z}^d}
\chi_{t\mathcal{D}}(k) \leq
N_{\varepsilon}(t+C\varepsilon),\label{3sides}
\end{equation}
where $C$ is a positive constant depending on $\mathcal{D}$. Thus by
using \eqref{asymp} and $\varepsilon=t^{-(d-1)/(d+1)}$ we get the
desired upper bound in \eqref{thm-bound1}, from which it is simple
to derive \eqref{thm-bound2}.

%%%%%%%%%%%%%%%%%%%%%%%%%%%%%%%%%%%%%%%%%%%%%%%%%%%%%%%%%%%%%%%%%%%%%%%%%%%%%%%%%%%

To prove the lower bound in \eqref{thm-bound1} (see also
\cite[P.167-168]{I-S-S-0}), we first apply the asymptotic expansion
in Schulz~\cite{schulz} to get
\begin{equation*}
\widehat{n_d d\sigma}(tk)=C_1 \sin(2\pi t
k_d-\pi\nu/2)(tk_d)^{-\nu}+O((tk_d)^{-\nu-1/\eta}),
\end{equation*}
where  $n_d$ is the $d^{\textrm{th}}$ component of the Gauss map of
$\partial \mathcal{D}$, $d\sigma$ is the induced Lebesgue measure on
$\partial \mathcal{D}$, $k=(0, \ldots, 0, k_d)$, $k_d\in
\mathbb{N}$, $\nu=\sum_{l=1}^{d-1}1/\omega_l$, $C_1$ is a constant
(depending on $\omega_1,\ldots, \omega_{d-1}$), and $\eta$ is the
least common multiple of  $\omega_1,\ldots, \omega_{d-1}$.

Hence by the Gauss--Green formula we can readily get an expansion of
$\widehat{\chi}_{\mathcal{D}}(tk)$. We then split the sum $S_{1, d}$
as follows
\begin{equation*}
S_{1, d}=t^d \sum_{\substack{k=(0, \ldots, 0, k_d)\\k_d\in
\mathbb{Z}_{*}^1}} \widehat{\chi}_{\mathcal{D}}(tk)+t^d
\sum_{\substack{k=(0, \ldots, 0, k_d)\\k_d\in \mathbb{Z}_{*}^1}}
\widehat{\chi}_{\mathcal{D}}(tk)(\widehat{\rho}(\varepsilon
k)-1)=:I+II
\end{equation*}
and apply the expansion. Therefore
\begin{equation*}
I=t^{d-1-\nu}g(t)+O(t^{d-1-\nu-1/\eta}),
\end{equation*}
where
\begin{equation*}
g(t)=C_2 \sum_{k_d\in \mathbb{Z}_{*}^1} |k_d|^{-\nu-1}\sin(2\pi t
|k_d|-\pi\nu/2)
\end{equation*}
(\footnote{This function $g$ matches results obtained by Kr{\"a}tzel
and Nowak, for example, \cite[P. 147]{K-N-2008}.}) with a positive
constant $C_2$ (depending on $\omega_1,\ldots, \omega_{d-1}$), and
\begin{equation*}
II=O(t^{d-1-\nu}(\varepsilon^{\nu}+\varepsilon)).
\end{equation*}

Note that the function $g$ is periodic and not identically zero,
hence $\limsup_{t\rightarrow \infty} |g(t)|>0$. Combining the two
bounds for I and II, the symmetry, \eqref{asymp}, and
\eqref{3sides}, we get the desired lower bound in
\eqref{thm-bound1}. This finishes the proof.

\subsection*{Acknowledgments}

This work was partially supported by the NSFC Grant No. 11501535 and
the Research Foundation of the University of Science and Technology
of China No. KY0010000027.

\end{document}